\documentclass[12pt]{article}
\usepackage{amsmath,amsxtra,amssymb,latexsym, amscd,amsthm}

\advance\voffset-1truecm\relax
\advance\hoffset-1truecm\relax

\def\R{{\Bbb R}}
\def\N{{\Bbb N}}
\def\R{{\Bbb R}}
\def\C{{\Bbb C}}
\def\P{{\Bbb P}}

\numberwithin{equation}{section}

\newcommand {\Cal}{\mathcal}

\newtheorem{lemma}{Lemma}[section]
\newtheorem{theorem}[lemma]{Theorem}

\def\leq{\leqslant}

\begin{document}
\title{{A property of the spherical derivative of an entire curve in complex projective space}}
\date{ Nguyen Thanh Son and Tran Van Tan} 
\maketitle
\begin{abstract}  We establish a  type of the Picard's theorem for entire curves in $P^n(\C)$ whose spherical derivative vanishes  on the inverse images of  hypersurface targets. Then, as a corollary, we prove that there is an union $D$ of finite number of hypersurfaces in the complex projective space $P^n(\C)$ such that for every entire curve $f$ in $P^n(\C)$, if  the spherical derivative
$f^{\#}$ of $f$  is bounded  on $ f^{-1}(D)$, then $f^{\#}$ is bounded on the entire complex plane, and hence, $f$ is a Brody curve.\\

\end{abstract}
\section{Introduction}
The five-point theorem of Lappan \cite{L1} states that a meromorphic function $ f$  in the unit disc
$\triangle:=\{z\in\C: |z|<1\}$  is normal if there are five distinct values $a_1,\dots,a_5$ such that
$$\sup\{(1-|z|^2)f^{\#}(z): z\in f^{-1}\{a_1,\dots,a_5\}\} < \infty.$$
The counterpart result for meromorphic functions  in $\C$ was given by Timoney (\cite{Ti}, Theorem 2): For a meromorphic function $f$ in $\C$, and for distinct five values $a_1,\dots,a_5$, if  $\sup\{f^{\#}(z): z\in f^{-1}\{a_1,\dots,a_5\}\} < \infty$ then $f^{\#}$ is upper bounded on $\C$.

 By using the Cartan-Nochka's second main theorem in Nevanlinna theory and techniques of uniqueness  problem, in \cite{Tan}, the second named author generalized these  results to the case of holomorphic mappings into $P^n(\C)$ and $n(2n+1)+2$ hyperplanes in general position in $P^n(\C)$.  For the case of hypersurfaces, there are several second main theorems have been established, such as in \cite{DT, DT1,ES, Q,QA,R2,R3}, however, they  are not enough for our purpose. In this paper, by obtaining  a  type of the Picard's theorem for entire curves in $P^n(\C)$ whose spherical derivative vanishes  on the inverse images of  hypersurface targets,  we extend the above mentioned results to the case of $3n\binom{n+d}{n}-n+1$ hypersurfaces in general position in $P^n(\C)$, $n\geq 2$, where $d$ is   the smallest  least common multiple of the degrees of these hypersurfaces. Because of limitation of techniques of the second main theorem for hypersurface targets, our following estimate on the number of targets is weaker than the one obtained by the second named author in \cite{Tan} for hyperplanes.

Let $f$ be  an entire curve in the complex projective space $P^n(\C).$ The spherical derivative $f^{\#}$ of $f$  measures the length distortion from the Euclidean metric in $\C$ to the Fubini-Study metric in $P^n(\C).$ The explicit formula  is
\begin{align*}
 f^{\#}=(|f_0|^2+\cdots+|f_n|^2)^{-2}\cdot\sum_{0\leq i<j\leq n}|f_if_j'-f_jf_i'|^2,
\end{align*}
where $(f_0,\dots,f_n)$ is a reduced representation of $f.$

An entire curve $f$ is called a Brody curve  if its spherical derivative is bounded. This is equivalent to normality of the family $\mathcal F:=\{f_a(z)=f(z+a): a\in\C\}.$

 The interested reader is referred  to \cite{D,E1,E2,MT,T,W} for many interesting results on Brody curves.

 Hypersurfaces $D_1,\dots,D_q$ $(q\geq n+1)$ in $P^n(\C)$ are said to be in general position if $\cap_{i=0}^nD_{j_i}=\varnothing,$ for all $1\leq j_0<\cdots<j_n\leq q.$

Our  main results can be stated as follows.
\begin{theorem}\label{SMT} Let $D_1,\dots, D_q$ be hypersurfaces  in general position in $P^n(\C)$, $n\geq 2$. Denote by $d$  the smallest  common multiple of  $\deg D_1,\dots, \deg D_q.$ Assume that there exists a  non-constant  entire curve  $f$   in $P^n(\C)$   such that for each $j\in\{1,\dots,q\},$  either $f(\C)\subset D_j$ or $f^{\#}=0$ on $f^{-1}(D_j).$ Then $q\leq 3n\binom{n+d}{n}-n$.
\end{theorem}
\begin{theorem}\label{Brody} Let $f$ be an entire curve in $P^n(\C)$, $n\geq 2.$ Assume that there are hypersurfaces $D_1,\dots,D_q$ in general position
in $P^n(\C)$  such that $f^{\#}$ is bounded on $\cup_{j=1}^q f^{-1}(D_j).$ Then $f$ is a Brody curve if
$q>3n\binom{n+d}{n}-n,$ where $d$ is  the smallest   common multiple of  $\deg D_1,\dots, \deg D_q.$
\end{theorem}

\noindent {\bf Acknowledgements:} This research was supported by the Vietnam National Foundation for Science and Technology Development (NAFOSTED).
\section{Notations}
Let $\nu$ be a nonnegative divisor on $\C.$ For each positive integer (or $+\infty) $ $ p,$ we define the
counting function of $\nu$ (where multiplicities are truncated by $p)$ by
\begin{align*}N^{[p]}(r, \nu) :=\int_{1}^r\frac{n_{\nu}^{[p]}}{t}dt
\quad (1 < r < \infty)
\end{align*}
where $n_{\nu}^{[p]}(t)=\sum_{|z|\leq t}\min\{\nu(z), p\}.$ For brevity we will omit the character $ [p]$  in the counting
function  if $p = +\infty.$

For a  meromorphic function $\varphi$  on $\C$, we denote by $(\varphi)_0$ the divisor of zeros of $\varphi$. We have the following Jensen's formula for the counting function:
\begin{align*} N(r, (\varphi)_0)-N(r, \left(\frac{1}{\varphi}\right)_0)=\frac{1}{2\pi}\int_{0}^{2\pi}\log\left|(\varphi(re^{i\theta})\right|d\theta +O(1).
\end{align*}
Let $f$ be a holomorphic mapping of $\C$ into $P^n(\C)$ with a reduced representation $f=(f_0,\dots,f_n).$ The characteristic  function $T_f(r)$ of $f$ is defined by
\begin{align*}T_f(r):=\frac{1}{2\pi}\int_{0}^{2\pi} \log\Vert f(re^{i\theta})\Vert d\theta-\frac{1}{2\pi}\int_{0}^{2\pi} \log \Vert f(e^{i\theta})\Vert d\theta,\quad r>1,
\end{align*}
where $\Vert f\Vert=\max\limits_{i=0,\dots,n}|f_i|.$

Let $D$ be a hypersurface in $P^n(\C)$ defined by a homogeneous polynomial $Q\in\C[x_0,\dots,x_n]$, $\deg Q=\deg D.$ Asumme that $f(\C)\not\subset D,$  then the counting function of $f$ with respect to $D$ is defined by  $N_f^{[p]}(r,D):=N^{[p]}(r, (Q(f_0,\dots, f_n))_0).$

Let $V \subset P^n(\C)$ be a projective variety of dimension $k.$ Denote by $I(V)$
the prime ideal in $\C[x_0, . . . , x_n]$ defining $V.$ Denote by $\C[x_0, . . . , x_n]_m$ the vector
space of homogeneous polynomials in $\C[x_0, . . . , x_n]$ of degree $m$ (including 0). Put
$I(V)_m := \C[x_0, . . . , x_n]_m \cap I(V).$

The Hilbert function $H_V$ of $V$ is defined by
$H_V(m) := \dim\frac{ \C[x_0, . . . , x_n]_m}{I(V)_m}.$

Consider two integer numbers $q, N$ satisfying   $q\geq N+1,$ $N\geq k.$
Hypersurfaces $D_1,\dots,D_q$  in $P^n(\C)$ are said to be in $N$-subgeneral position with respect to $V$ if $V\cap(\cap_{i=0}^ND_{j_i})=\varnothing,$ for all $1\leq j_0<\cdots<j_N\leq q.$

\section{Proof of Theorems}
\begin{lemma}[\cite{AK}, Lemma 3.1] \label{Zalcman}Let $\mathcal F$ be a family of holomorphic mappings of $\C$ into $P^n(\C).$ If $\mathcal F$ is not normal then there exist  sequences $\{z_k\}\subset \C$ with $z_k\to z_0\in \C,$  $\{f_k\}\subset\mathcal F$, $\{\rho_k\}\subset\R$ with $\rho_k\to 0^+,$  such that $g_k(\zeta) := f_ k(z_k + \rho_k \zeta)$  converges uniformly on compact subsets of $\C$ to a nonconstant
holomorphic mapping $g$ of $\C$ into $P^n(\C).$
\end{lemma}
\begin{lemma}[\cite{Fu}, Theorem 2.4.11]\label{N1} Let $\Cal S$ be a complex  vector space of dimension $k+1.$ Consider positive integer numbers $N,q$  satisfying $N\geq k,  q\geq 2N-k+1.$ Let $v_1,\dots, v_q$ be non-zero vectors in $\Cal S.$ Assume that every set of $N+1$ vectors in $\{v_1,\dots,v_q\}$ has  rank  $k+1.$  Then, there exist constants $\omega_1,\dots,\omega_q$ and $\Theta$ satisfying the following conditions:

(i) $\quad 0<\omega_j\leq\Theta\leq 1$ for all $j\in\{1,\dots,q\};$

(ii) $\quad\sum_{j=1}^q\omega_j\leq \Theta(q-2N+k-1)+k+1;$

(iii) $\quad\frac{k+1}{2N-k+1}\leq \Theta\leq \frac{k+1}{N+1};$

(iv)$\quad$ if $R\subset \{1,\dots,q\}$ and $\#R= N+1,$ then $\sum_{j\in R}\omega_j\leq k+1.$
\end{lemma}
We call constants $\omega_j$ $(1\leq j\leq q)$ and $\Theta$ satisfying properties (i) to (iv) in Lemma \ref{N1} Nochka weights and Nochka constant associated to vectors $v_j$'s.
\begin{lemma}[\cite{Fu}, Proposition 2.4.15] \label{N2} Let $\Cal S$ be a complex  vector space of dimension $k+1.$ Consider positive integer numbers $N,q$  satisfying $N\geq k,  q\geq 2N-k+1.$ Let $v_1,\dots, v_q$ be non-zero vectors in $\Cal S.$ Assume that every set of $N+1$ vectors in $\{v_1,\dots,v_q\}$ has  rank  $k+1.$ Let $\omega_1,\dots,\omega_q$ be Nochka weights for $v_1,\dots, v_q,$ respectively.  Consider   arbitrary non-negative real constants $E_1,\dots, E_q$. Then for each  subset $R$ of $\{1,\dots, q\}$ with $\#R= N+1$, there exists $R'\subset R$ such that $\{v_j, j\in R'\}$ is a basis of $\Cal S$ and
\begin{align*}\sum_{j\in R}\omega_jE_j\leq\sum_{j\in R'}E_{j}.
\end{align*}
\end{lemma}

\begin{proof}[Proof of Theorem \ref{SMT}]  Denote by $V$ the smallest algebraic  variety in $P^n(\C)$ containing $f(\C),$ then $k:=\dim V\geq 1.$

 Without loss of the generality, we may assume that $f(\C)\not\subset D_j$ for all $j\in\{1,\dots,q_0\}$ and $f(\C)\subset D_j$ for all $j\in\{q_0+1,\dots,q\},$ for some $q_0\leq q.$ Since $D_1,\dots,D_q$ are in general position in $P^n(\C)$,  we have that $q-q_0+k\leq n,$ and
  $D_1, \dots, D_{q_0}$ are  in $n-(q-q_0)$-subgeneral position with respect to  $V.$

   For each $j\in\{1,\dots,q_0\},$ let $G_j\in\C[x_0,\dots,x_n]$   be a homogeneous polynomial defining $D_j, $ $\deg G_j=\deg D_j,$ and set $Q_j=G_j^{\frac{d}{\deg D_j}}.$ Then  $\deg Q_j=d,$ for all $j\in\{1,\dots,q_0\}.$
Let $(f_0,\dots, f_n)$ be a reduced representation of $f,$  and set $Q_j(f):=Q_j(f_0,\dots,f_n).$

\noindent We have
\begin{align} \label{223}N^{[p]}(r, (Q_j(f))_0)\leq \frac{d}{\deg D_j}N_f^{[p]}(r,D_j),
\end{align}
where $p$ is an integer number or infinity.

Set $\Cal V:= \frac{\C[x_0,\dots,x_n]_d}{I(V)_d}.$ Since $V\not\subset D_j,$ we have that $Q_1,\dots,Q_{q_0}$ are  non-zero vectors in $\Cal V.$
 For each $R\subset\{1,\dots,q_0\},$ $\#R=n+1-(q-q_0),$ since $(\cap_{j\in R}D_j)\cap V=\varnothing$ we get that
\begin{align}\label{1.6}
\{Q_j, j\in R\}\; \text{has rank not less than}\; k+1\; \text{ in the}\; \C\text{-vector space}\; \Cal V.
\end{align}
Denote by $\mathcal E$ the set of all subsets $J\subset\{1,\dots,q_0\}$ such that $1\leq\# J\leq k+1$ and  vectors $Q_{j}$, $j\in J$ are linearly independent in $\Cal V.$ Then $\cup _{E\in\Cal E}E=\{1,\dots,q_0\}.$

\noindent It is easily to see that, there exist  a subset $\{v_1,\dots,v_{H_V(d)-k-1}\}\subset \C[x_0,\dots,x_n]_d$ such that for all  $J\in \Cal E,$ vectors $v_1,\dots,v_{H_V(d)-k-1}, Q_j, j\in J$ are linearly independent in $\Cal V$. Here, if $H_V(d)=k+1,$ we choose $\{v_1,\dots,v_{H_V(d)-k-1}\}=\varnothing.$
 Denote by $\big< v_1,\dots,v_{H_V(d)-k-1}\big>$  the subspace of $\Cal V$  which is generated  by $v_1,\dots,v_{H_V(d)-k-1}.$ Then $Q_1,\dots, Q_{q_0}$ are non-zero vectors in the $(k+1)$-dimensional vector space $\frac{\Cal V}{\big< v_1,\dots,v_{H_V(d)-k-1}\big>}.$ Furthermore, by (\ref{1.6}), for each $R\subset\{1,\dots,q_0\}$ with $\#R=n+1-(q-q_0)$, there exists $R'\subset R,$ $\#R'=k+1$ such that $Q_j,j\in R'$ form a basis in  $\frac{\Cal V}{\big< v_1,\dots,v_{H_V(d)-k-1}\big>}.$

Applying Lemmas \ref{N1}-\ref{N2}, there are Nochka weights and the  Nochka constant $\omega_1,\dots, \omega_{q_0},\Theta$ associated  to vectors $Q_1,\dots,Q_{q_0}$ in $\frac{\Cal V}{\big< v_1,\dots,v_{H_V(d)-k-1}\big>}.$

We now consider polynomials   $P_1,\dots,P_{H_V(d)}$ in $\C[x_0,\dots,x_n]_d$ such that they form  a basis in $\Cal V=\frac{\C[x_0,\dots,x_n]_d}{I(V)_d}.$


  Denote by $W$ the Wronskian of $P_1(f_0,\dots,f_n),\dots, P_{H_V(d)}(f_0,\dots,f_n)$. Since $f$ is algebraically non-degenerate, we have $W\not\equiv 0.$

  Since hypersurfaces $D_1,\dots, D_{q_0}$ are in $n-(q-q_0)-$subgeneral position with respect to $V$, there is a positive constant $c$ such that for all  $J\subset\{1,\dots,q\}$, $\# J=n+1-(q-q_0),$ and for all $z\in\C,$ we have
$$\max_{j\in J} |Q_j(f(z)|\geq c\Vert f(z)\Vert^d.$$
For each $z\in \C$, we take $K_z\subset \{1,\dots,q_0\},$ $\#K_z=q_0-n-1$ such that $|Q_j(f(z))|\geq c\Vert f(z)\Vert^d$ for all $j\in K_z,$  and put $J_z:=\{1,\dots,q_0\}\setminus K_z.$
Hence, there are positive constants $c_1,c_2$ such that for all $z\in \C$
\begin{align}\label{22.10}
\log\frac{\prod_{j=1}^{q_0}|Q_j(f(z))|^{\omega_j}}{|W(z)|}&=\log\prod_{j\in K_z}|Q_j(f(z))|^{\omega_j}-\log|W(z)|\notag\\
&\quad+\log\prod_{j\in J_z}|Q_j(f(z))|^{\omega_j}\notag\\
&\geq(\omega_1+\cdots+\omega_{q_0})\log\Vert f(z)\Vert^d -\log|W(z)|-c_1\notag\\
&\quad-\log\frac{\prod_{j\in J_z}\Vert f(z)\Vert^{d\omega_j}}{\prod_{j\in J_z}|Q_j(f(z))|^{\omega_j}}\notag\\
&\geq d(\omega_1+\cdots+\omega_{q_0})\log\Vert f(z)\Vert -\log|W(z)|-c_2\notag\\
&\quad-\sum_{j\in J_z}\log\left(\frac{\Vert f(z)\Vert^d\Vert Q_j\Vert}{|Q_j(f(z))|}\right)^{\omega_j}
\end{align}
where $\Vert Q_j\Vert$ is the sum of absolute values of coefficients of $Q_j.$

\noindent On the other hand, by Lemma \ref{N2}, there is a subset $T_z\subset J_z,$ $\#T_z=k+1$ such that the polynomials  $Q_j (j\in T_z)$ form a basis in  $\frac{\Cal V}{\big< v_1,\dots,v_{H_V(d)-k-1}\big>},$  and
$$\sum_{j\in J_z}\log\left(\frac{\Vert f(z)\Vert^d\Vert Q_j\Vert}{|Q_j(f(z))|}\right)^{\omega_j}\leq \sum_{j\in T_z}\log\frac{\Vert f(z)\Vert^d\Vert Q_j\Vert}{|Q_j(f(z))|}.$$
Since $P_1,\dots, P_{H_V(d)}$ form a basis of $\frac{\C[x_0,\dots,x_n]_d}{I(V)_d}$ and $Q_j, j\in T_z$ are linearly independent in $\frac{\C[x_0,\dots,x_n]_d}{I(V)_d},$  there is a subset $\tau_z\subset\{1,\dots,H_V(d)\}$ with  $\#\tau_z=H_V(d)-(k+1)$ such that $ P_i, Q_j\; (i\in\tau_z, j\in T_z)$ form a basis of $\frac{\C[x_0,\dots,x_n]_d}{I(V)_d}.$
Hence, by (\ref{22.10}) there is a positive constant $c_3$ such that for all $z\in \C$
\begin{align}
\log\frac{\prod_{j=1}^{q_0}|Q_j(f(z))|^{\omega_j}}{|W(z)|}&\geq d(\omega_1+\cdots+\omega_{q_0})\log\Vert f(z)\Vert - \log|W(z)|-c_2\notag\\
&\quad-\sum_{j\in T_z}\log\frac{\Vert f(z)\Vert^d\Vert Q_j\Vert}{|Q_j(f(z))|}\notag\\
&\geq d(\omega_1+\cdots+\omega_{q_0})\log\Vert f(z)\Vert -c_3\notag\\
&\quad-\log\frac{|W(z)|\cdot\Vert f(z)\Vert^{dH_V(d)}}{\prod_{j\in T_z}|Q_j(f(z))|\cdot\prod_{j\in \tau_z}|P_j(f(z))|}\notag\\
&\geq d(\omega_1+\cdots+\omega_{q_0}-H_V(d))\log\Vert f(z)\Vert-c_3\notag\\
&\quad-\sum_{(T,\tau)}\log^+\frac{|W(z)|}{\prod_{j\in T}|Q_j(f(z))|\cdot \prod_{i\in \tau}|P_i(f(z))|},\label{28.1}
\end{align}
where the last sum is taken over all pairs $(T,\tau) $  satisfying the following conditions:

i) $T\subset \{1,\dots, q_0\}, $ $\#T= k+1,$ and $Q_j, j\in T$ are linearly independent  in $\frac{\C[x_0,\dots,x_n]_d}{I(V)_d};$

ii) $\tau\subset\{1,\dots,H_V(d)\},$ $\#\tau=H_d(V)-(k+1)$ and $P_i, Q_j$  $( i\in \tau, j\in T$) form a basis of $\frac{\C[x_0,\dots,x_n]_d}{I(V)_d}.$

Applying integration on both sides of (\ref{28.1}), using Lemma \ref{N1},  the Jensen's Lemma and the Logarithmic derivative Lemma, we have
\begin{align}\Big\Vert&\sum_{j=1}^{q_0}\omega_j N(r,(Q_j(f))_0)-N(r, (W)_0)\notag\\
&\quad\quad\geq d\left(\Theta[q_0-2(n-(q-q_0))+k-1]+k+1-H_V(d)\right)T_f(r)-o(T_f(r))\notag\\
&\quad\quad\geq d\Theta\left(q_0-\frac{[2(n-(q-q_0))-k+1]H_V(d)}{(k+1)}\right)T_f(r)-o(T_f(r)).\label{28.2}
\end{align}
Here, as usual, by the notation $“\Big\Vert P”$ we mean the assertion $P $ holds for all $ r \in[1, +\infty)$
excluding a Borel subset $E$ of $(1, +\infty )$ with $\int_{E}dr < +\infty.$

In order to estimate  the left side of (\ref{28.2}), we consider an arbitrary point $a\in \cup_{j=1}^{q_0}f^{-1}(D_j).$

 Since $D_1,\dots,D_{q_0}$ are in $n-(q-q_0)$-subgeneral position, there is a subset $R_a\subset\{1,\dots,q_0\}$, $\# R_a=n+1-(q-q_0)$, such that $f(a)\not\in D_j$ for all $j\in\{1,\dots,q_0\}\setminus R_a.$ Hence,
\begin{align}(Q_j(f))_0(a)=0\label{22.10.2}
\end{align}
for all $j\in\{1,\dots,q_0\}\setminus R_a.$

By Lemma \ref{N2}, there is a subset $\{j_{0},\dots,j_{k}\}\subset R_a$ such that  $Q_{j_{0}},\dots, Q_{j_{k}}$ form a basis in $\frac{\Cal V}{\big< v_1,\dots,v_{H_V(d)-k-1}\big>}$ and
\begin{align}&\sum_{j=1}^{q_0}\omega_j\max\{(Q_j(f))_0(a)-H_V(d)+1,0\}\notag\\
&\quad=\sum_{j\in R_a}\omega_j\max\{(Q_j(f))_0(a)-H_V(d)+1,0\}\notag\\
&\quad\leq\sum_{i=0}^k\max\{(Q_{j_i}(f))_0(a)-H_V(d)+1,0\} \label{28.3}.
\end{align}
Denote by $W_1$ the Wronskian of $v_1(f_0,\dots, f_n),\dots,v_{H_V(d)-k-1}(f_0,\dots, f_n)$, $ Q_{j_0}(f_0,\dots, f_n),\dots,Q_{j_{k}}(f_0,\dots, f_n).$ Since $ v_1,\dots,v_{H_V(d)-k-1}, Q_{j_0},\dots,Q_{j_k}$ form a basis of $\Cal V=\frac{\C[x_0,\dots,x_n]_d}{I(V)_d},$ there is a non-zero constant $c$ such that $$W_1=cW.$$
By the assumption, $f^{\#}(a)=0$, and hence,
\begin{align}(f_0(a):\cdots:f_n(a))=(f_0'(a):\cdots:f'_n(a)). \label{2201}
\end{align}

We distinguish two cases.

\noindent{\bf Case 1.} $H_V(d)=2$ (then $d=k=1$ and $W_1$ is the Wronskian of $Q_{j_0}(f),Q_{j_1}(f)).$

For each $j\in\{1,\dots,q_0\},$  if $Q_{j}(f(a))=0$ then  by (\ref{2201}) we have $(Q_{j}(f))'(a)=0.$
 Therefore, for each $j\in\{1,\dots,q_0\},$ we have
\begin{align}\label{242}
(Q_{j}(f))_0(a)=0 \;\text{or}\; (Q_{j}(f))_0(a)\geq 2.
\end{align}
By  (\ref{22.10.2}),  (\ref{242}),  and Lemma \ref{N1}, $(iv),$  for any  $a\in \cup_{j=1}^{q_0}f^{-1}(D_j),$ we have
\begin{align*}\sum_{j=1}^{q_0}&\omega_j(Q_j(f))_0(a)-(W)_0(a)\notag\\
&=\sum_{j\in R_a}\omega_j(Q_j(f))_0(a)-(W_1)_0(a)\notag\\
&\leq\sum_{j\in R_a}\omega_j(Q_j(f))_0(a)-\sum_{i=0,1}\max\{(Q_{j_i}(f))_0(a)-1,0\}\notag\\
&\leq\sum_{j\in R_a}\omega_j(Q_j(f))_0(a)-\sum_{j\in R}\omega_j\max\{(Q_j(f))_0(a)-1,0\}\notag\\
&=\sum_{j\in R_a}\omega_j\min\{(Q_j(f))_0(a),1\}\notag\\
&\leq \sum_{j\in R_a}\frac{\omega_j}{2}(Q_j(f))_0(a)\notag\\
&= \sum_{j=1}^{q_0}\frac{\omega_j}{2} (Q_j(f))_0(a).
\end{align*}
Hence, since $d=1$, we have
\begin{align}
\sum_{j=1}^{q_0}\omega_j N(r,(Q_j(f))_0)-N(r, (W)_0)&\leq \sum_{j=1}^{q_0}\frac{\omega_j}{2}N(r,  (Q_j(f))_0)\notag\\
&\leq\sum_{j=1}^{q_0}\frac{\omega_j}{2}T_f(r) +O(1).\label{08.20}
\end{align}
By (\ref{28.2}), (̣\ref{08.20}) and since $d=k=1, H_V(d)=2$, $ \Theta\geq \omega_j$, we have
\begin{align*}
\Big\Vert\;\left(q_0-\frac{4(n-(q-q_0))}{2}\right)T_f(r)-o(T_f(r))\leq\frac{q_0}{2}T_f(r).
\end{align*}
Therefore, $2q_0-4(n-(q-q_0))\leq q_0,$ hence, $q\leq 4n-3(q-q_0)\leq4n<3n\binom{n+d}{n}-n.$

\noindent {\bf Case 2.} $H_V(d)\geq 3.$

For the sake of convenience, we set $v_{H_V(d)-k+i}:=Q_{j_i},$ $i=0,1,\dots, k.$ Then, by (\ref{2201}) and  the Euler's formula, we have
\begin{align}\label{220}&\left((v_1(f))'(a):\cdots:(v_{H_V(d)}(f))'(a)\right)\notag\\
&\quad=\left(\sum_{s=0}^n\frac{\partial v_{1}}{\partial x_s}(f(a))\cdot f_s'(a):\cdots:\sum_{s=0}^n\frac{\partial v_{H_V(d)}}{\partial x_s}(f(a))\cdot f_s'(a)\right)\notag\\
&\quad=\left(\sum_{s=0}^n\frac{\partial v_{1}}{\partial x_s}(f(a))\cdot f_s(a):\cdots:\sum_{s=0}^n\frac{\partial v_{H_V(d)}}{\partial x_s}(f(a))\cdot f_s(a)\right)\notag\\
&\quad=\left(v_1(f(a)):\cdots:v_{H_V(d)}(f(a))\right).
\end{align}
Then
\begin{align}
(W_1)_0(a)\geq 1.\label{29.3}
\end{align}
By the Laplace expansion Theorem, we have
\begin{align}
W_1
&=\begin{vmatrix}v_1(f)&v_2(f)&\cdots&v_{H_V(d)}(f) \\ (v_1(f))'&(v_2(f))'&\cdots&(v_{H_V(d)}(f))' \\
\cdot&\cdot&\cdots&\cdot\\
\cdot&\cdot&\cdots&\cdot\\
\cdot&\cdot&\dots&\cdot\\
(v_1(f))^{(H_V(d)-1)}&(v_2(f))^{(H_V(d)-1)}&\cdots&(v_{H_V(d)}(f))^{(H_V(d)-1)}\end{vmatrix}\notag\\
&=\sum_{1\leq t<\ell\leq H_V(d)}(-1)^{t+\ell} \left(v_t(f)(v_\ell(f))'-v_\ell(f)(v_t(f))'\right) \det W_{t,\ell}\label{29.1}
\end{align}
where $W_{(t,\ell)}$ is the matrix which is defined from the  matrix $\left(v_i(f)^{(s)}\right)_{1\leq s+1,i\leq H_V(d)}$ by omitting two first rows and $t^{th}$, $\ell^{th}$ columns.

\noindent For each $1\leq t<\ell\leq H_V(d),$ it is clear that
\begin{align}
(\det W_{(t,\ell)})_0(a)\geq\sum_{i\in\{1,\dots,H_V(d)\}\setminus\{t,\ell\}}^{H_V(d)}\max\{(v_i(f))_0(a)-H_V(d)+1,0\}.\label{29.2}
\end{align}
We now prove that for all $1\leq t\ne\ell\leq H_V(d),$
\begin{align}\left(v_t(f)v'_\ell(f)-v_\ell(f)v'_t(f)\right)_0(a)\geq &\max\{(v_t(f))_0(a)-H_V(d)+1,0\}\notag\\
&+\max\{(v_\ell(f))_0(a)-H_V(d)+1,0\}+1.\label{29.4}
\end{align}

If $ (v_t(f))_0(a)\leq H_V(d)-1, (v_\ell(f))_0(a)\leq H_V(d)-1,$
then, the right side of (\ref{29.4}) is equal to 1, but by (\ref{29.3}), the left side of (\ref{29.4}) is not less than 1.

If $ (v_t(f))_0(a)\geq H_V(d)$ or $(v_\ell(f))_0(a)\geq H_V(d),$ without loss of generality, we assume that $ (v_t(f))_0(a)\geq H_V(d).$  Since, $H_V(d)\geq 3,$ we have
\begin{align*}
&\left(v_t(f)v'_\ell(f)-v_\ell(f)v'_t(f)\right)_0(a)\\
&\quad \geq [(v_t(f))_0(a)-H_V(d)+1]+(v_\ell(f))_0(a)+1\\
&\quad =\max\{(v_t(f))_0(a)-H_V(d)+1,0\}+(v_\ell(f))_0(a)+1\\
 &\quad\geq\max\{(v_t(f))_0(a)-H_V(d)+1,0\}+\max\{(v_\ell(f))_0(a)-H_V(d)+1,0\}+1.
\end{align*}
Hence, we get (\ref{29.4}).

\noindent By (\ref{29.1}), (\ref{29.2}) and (\ref{29.4}),  we have
\begin{align}(W)_0(a)&=(W_1)_0(a)\notag\\
&\geq \sum_{s=1}^{H_V(d)}\max\{(v_s(f))_0(a)-H_V(d)+1,0\}+1\notag\\
&\geq\sum_{s=H_V(d)-k}^{H_V(d)}\max\{(v_s(f))_0(a)-H_V(d)+1,0\}+1\notag\\
&\geq\sum_{i=0}^k\max\{(Q_{j_i}(f))_0(a)-H_V(d)+1,0\}+1. \label{29.5}
\end{align}
By  (\ref{22.10.2}), (\ref{28.3}), (\ref{29.5}), and Lemma \ref{N1}, $(iv),$  for each $a\in\cup_{j=1}^{q_0}f^{-1}(D_j),$ we have
\begin{align*}\sum_{j=1}^{q_0}&\omega_j(Q_j(f))_0(a)-(W)_0(a)\notag\\
&\leq\sum_{j\in R_a}\omega_j(Q_j(f))_0(a)-\sum_{i=0}^k\max\{(Q_{j_i}(f))_0(a)-H_V(d)+1,0\}-1\notag\\
&\leq\sum_{j\in R_a}\omega_j(Q_j(f))_0(a)-\sum_{j\in R}\omega_j\max\{(Q_j(f))_0(a)-H_V(d)+1,0\}-1\notag\\
&=\sum_{j\in R_a}\omega_j\min\{(Q_j(f))_0(a),H_V(d)-1\}-1\notag\\
&\leq \sum_{j\in R_a}\omega_j\left(\min\{(Q_j(f))_0(a),H_V(d)-1\}-\frac{1}{k+1}\min\{(Q_j(f))_0(a),1\}\right)\notag\\
&\leq \sum_{j\in R_a}\omega_j\left(1-\frac{1}{(k+1)\left(H_V(d)-1\right)}\right)\min\{(Q_j(f))_0(a),H_V(d)-1\}\notag\\
&= \sum_{j=1}^{q_0}\omega_j\left(1-\frac{1}{(k+1)\left(H_V(d)-1\right)}\right) \min\{(Q_j(f))_0(a),H_V(d)-1\}.
\end{align*}
Hence,
\begin{align*}&\sum_{j=1}^{q_0}\omega_j N(r,(Q_j(f))_0)-N(r, (W)_0)\\
&\quad\leq \sum_{j=1}^{q_0}\omega_j\left(1-\frac{1}{(k+1)\left(H_V(d)-1\right)}\right) N^{[H_V(d)-1]}(r,(Q_j(f))_0).
\end{align*}
Combining with (\ref{223}), (\ref{28.2}) and by the fact that $\omega_j\leq \Theta$ (Lemma \ref{N1}), we get
\begin{align*}
\Big\Vert &d\left(q_0-\frac{(2(n-(q-q_0))-k+1)H_V(d)}{(k+1)}\right)T_f(r)\\
&\quad\leq \sum_{j=1}^{q_0}\left(1-\frac{1}{(k+1)\left(H_V(d)-1\right)}\right) N^{[H_V(d)-1]}(r,(Q_j(f))_0) +o(T_f(r))\\
&\quad\leq \sum_{j=1}^{q_0}\left(1-\frac{1}{(k+1)\left(H_V(d)-1\right)}\right)\frac{d}{\deg D_j} N_f^{[H_V(d)-1]}(r,D_j) +o(T_f(r)).
\end{align*}
Then
\begin{align}
\Big\Vert &\left(q_0-\frac{(2(n-(q-q_0))-k+1)H_V(d)}{(k+1)}\right)T_f(r)\notag\\
&\quad\leq \left(1-\frac{1}{(k+1)\left(H_V(d)-1\right)}\right)\sum_{j=1}^{q_0}\frac{1}{\deg D_j} N_f^{[H_V(d)-1]}(r,D_j) +o(T_f(r)).\label{8.22}
\end{align}

  For each $z\in\cup_{j=1}^{q_0}f^{-1}(D_j)$, we define
$$\Cal C_z:=\{(c_0,\dots,c_n)\in\C^{n+1}: c_0f_0(z)+\cdots c_nf_n(z)=0\}.$$
Take a point $(b_0:\cdots:b_n)\in V$ and set $\Cal B_b:=\{(c_0,\dots,c_n)\in\C^{n+1}: c_0b_0+\cdots+c_nb_n=0\}.$
Since $\Cal C_z, \Cal B_b$ are vector subspaces of dimension $n$ in $\C^{n+1}$ and since $\cup_{j=1}^{q_0}f^{-1}(D_j)$  is at most countable, it follows that there exists
\begin{equation*}
(c_0,\dots, c_n)\in \C^{n+1}\setminus\left((\cup_{z\in\cup_{j=1}^q f^{-1}(D_j)}\Cal C_z)\cup\Cal B_b\right).
\end{equation*}
Set $\gamma_0:=c_0x_0+\cdots+\gamma_nx_n\in\C[x_0,\dots,x_n]_1.$ Then $\gamma_0\not\equiv 0$ on $V,$ and hence, there are $\gamma_1,\dots,\gamma_k$ in   $\C[x_0,\dots,x_n]_1$ such that $\gamma_0,\dots, \gamma_k$ have no common zero points in $V$ (note that $V$ is irreducible and $\dim V=k$).

 By our choice for $ (c_0,\dots, c_n), $
 \begin{align}
\{z: \gamma_0(f_0(z),\dots,f_n(z))=0\}\cap (\cup_{j=1}^{q_0}f^{-1}(D_j))=\varnothing.\label{t1}
\end{align}
Set $F:=(\gamma_0(f):\cdots: \gamma_{k}(f)):\C\to\P^k(\C).$ Then $F$ is linearly non-degenerate and $T_F(r)=T_f(r)+O(1).$

Since $f^{\#}$ vanishes on $ \cup_{j=1}^{q_0}f^{-1}(D_j)$, we have $$(f_0:\cdots:f_n)=(f'_0:\cdots:f'_n)\;\text{ on }\; \cup_{j=1}^{q_0}f^{-1}(D_j).$$
Hence, for each $z\in\cup_{j=1}^{q_0}f^{-1}(D_j)$, we have
\begin{align}(\gamma_0(f(z)):\cdots: \gamma_k(f(z)))=\left((\gamma_0(f))'(z):\cdots: (\gamma_k(f))'(z)\right).\label{t2}
\end{align}
Since    $F$ is  linearly nondegenerate,  there exists $t\in\{1,\dots,k\}$ such that
\[\det\begin{pmatrix}\gamma_0(f)&\gamma_{t}(f)\\
(\gamma_0(f))'&(\gamma_{t}(f))'\end{pmatrix}\not\equiv 0,\;\;\text{hence,}\left(\frac{\gamma_{t}(f)}{\gamma_0(f)}\right)'\not\equiv 0 .\]
By (\ref{t1}) and (\ref{t2}), we have
\begin{align}
\left(\frac{\gamma_{t}(f)}{\gamma_0(f)}\right)'=0\;\text{ on}\; \cup_{j=1}^{q_0}f^{-1}(D_j).\label{t3}
\end{align}
From the first main theorem and the lemma on logarithmic derivative of Nevanlinna theory for meromorphic functions, we get easily  that
\begin{align*}T_{\left(\frac{\gamma_{t}(f)}{\gamma_0(f)}\right)'}(r)&\leq 2T_{\left(\frac{\gamma_{t}(f)}{\gamma_0(f)}\right)}(r)+o\left(T_{\left(\frac{\gamma_{t}(f)}{\gamma_0(f)}\right)}(r)\right)\\
&\leq 2T_F(r)+o\left(T_F(r)\right)\\
&= 2T_f(r)+o\left(T_f(r)\right).
\end{align*}
On the other hand, for each $z_0\in\C,$ since $D_1,\dots, D_{q_0}$ are in $n-(q-q_0)$-subgeneral position in $V$, it follows that   there are at most $n-(q-q_0)$  of them passing  through  $f(z_0).$  Hence, by (\ref{t1}) and (\ref{t3}), we have
\begin{align*}\sum_{j=1}^{q_0}N^{[1]}_f(r,D_j)&\leq (n-q+q_0) N_{\left(\frac{\gamma_{t}(f)}{\gamma_0(f)}\right)'}(r)\notag\\
&\leq (n-q+q_0) T_{\left(\frac{\gamma_{t}(f)}{\gamma_0(f)}\right)'}(r)+O(1)\notag\\
&=2(n-q+q_0)T_f(r)+o(T_f(r)).
\end{align*}
Combining with (\ref{8.22}), we have
\begin{align*}
\Big\Vert &\left(q_0-\frac{(2(n-q+q_0)-k+1)H_V(d)}{k+1}\right)T_f(r)\\
&\quad\leq\left(1-\frac{1}{(k+1)(H_V(d)-1)}\right)\sum_{j=1}^{q_0}\frac{1}{\deg D_j}N_f^{[H_V(d)-1]}(r,D_j)+o(T_f(r))\\
&\quad\leq\left(H_V(d)-1-\frac{1}{(k+1)}\right)\sum_{j=1}^{q_0}\frac{1}{\deg D_j}N_f^{[1]}(r,D_j)+o(T_f(r))\\
&\quad\leq2(n-q+q_0)\left(H_V(d)-1-\frac{1}{(k+1)}\right)T_f(r)+o(T_f(r)).
\end{align*}
Therefore,
\begin{align*}
q &=q_0+(q-q_0)\\
&\leq\frac{(2(n-q+q_0)-k+1)H_V(d)}{k+1}+2(n-q+q_0)\left(H_V(d)-1-\frac{1}{(k+1)}\right)\\
&=(2n-1)H_V(d)+\frac{2(n+1)H_V(d)-2n}{k+1}\\
&\quad\quad\quad-(q-q_0)\left(\frac{2H_V(d)-2}{k+1}+2(H_V(d)-1)\right)\\
&\leq(2n-1)H_V(d)+\frac{2(n+1)H_V(d)-2n}{k+1}\\
&\leq(2n-1)H_V(d)+\frac{2(n+1)H_V(d)-2n}{2}\\
&\leq3n\binom{n+d}{n}-n.
\end{align*}
\end{proof}

\begin{proof}[Proof of Theorem \ref{Brody}]
Let $(f_0,\dots,f_n)$ be a reduced representation of $f.$ Assume that $f$ is not a Brody curve. Then the family $\mathcal F:=\{f_a(z):=f(a+z): a\in\C\}$ is not normal.  By Lemma \ref{Zalcman}, there exist sequences $\{z_k\}\subset \C$ with $z_k\to z_0\in \C,$  $\{a_k\}\subset\C$, $\{\rho_k\}\subset\R$ with $\rho_k\to 0^+,$  such that $g_k(\zeta) := f_ {a_k}(z_k + \rho_k \zeta)=f(a_k+z_k+\rho_k\zeta)$  converges uniformly on compact subsets of $\C$ to a nonconstant
holomorphic mapping $g$ of $\C$ into $P^n(\C).$

For each $j_0\in\{1,\dots,q\}$ satisfying  $g(\C)\not\subset D_{j_0}$, we  now prove that $g^{\#}(\xi)=0$ for all $\xi\in g^{-1}(D_{j_0}).$
To do this, we consider an arbitrary point  $\xi_0\in g^{-1}(D_{j_0}).$    By Hurwitz's Theorem there are values $\{\xi_k\}$  (for all $k$ sufficiently large), $\xi_k\to\xi_0$ such that $\xi_k\in g_k^{-1}(D_{j_0}),$ and hence, $a_k+z_k+\rho_k\xi_k\in f^{-1}(D_{j_0}).$ Then by the assumption, there is a positive constant $M$ such that, for all $k$ sufficiently large,
\begin{align*}f^{\#}(a_k+z_k+\rho_k\xi_k)\leq M.
\end{align*}
Then
\begin{align*}g^{\#}(\xi_0)&=\lim_{k\to\infty}g_k^{\#}(\xi_k)\\
&=\lim_{k\to\infty} \rho_kf^{\#}(a_k+z_k+\rho_k\xi_k)\\
&=0.
\end{align*}
Hence, for each $j\in\{1,\dots,q\}$, either $g(\C)\subset D_j$ or $g^{\#}=0$ on $g^{-1}(D_j).$ Therefore, since $n\geq 2$ and $q>(2dn+n)\binom{n+d}{n}-dn$, Theorem \ref{SMT} shows that $g$ is a constant curve; this is impossible.
\end{proof}

\noindent Nguyen Thanh Son\\
Department of Mathematics\\
  Hanoi National University of Education\\
 136-Xuan Thuy street, Cau Giay, Hanoi, Vietnam\\
e-mail:  k16toannguyenthanhson@gmail.com\\

\noindent Tran Van Tan\\
Department of Mathematics\\
  Hanoi National University of Education\\
 136-Xuan Thuy street, Cau Giay, Hanoi, Vietnam\\
e-mail: tantv@hnue.edu.vn
\end{document}